\documentclass[draft]{amsart}
\usepackage{amssymb,amsmath}
\usepackage{amsthm}
\usepackage{mathrsfs}
\usepackage{enumerate, indentfirst}
\usepackage[fleqn,tbtags]{mathtools}
\usepackage{multicol}
 \newtheorem{theorem}{Theorem}[section]
 \newtheorem{corollary}[theorem]{Corollary}
 \newtheorem{lemma}[theorem]{Lemma}

 \theoremstyle{definition}

 \theoremstyle{remark}
 \newtheorem{remark}[theorem]{Remark}
  \numberwithin{equation}{section}

\renewcommand{\phi}{\varphi}
\renewcommand{\theta}{\vartheta}

\DeclareMathOperator{\tform}{\mathfrak{t}}

\DeclareMathOperator{\wform}{\mathfrak{w}}

\DeclarePairedDelimiterX\sipt[2]{(}{)_{\tform}}{#1\,\delimsize\vert\,#2}
\DeclarePairedDelimiterX\sipv[2]{(}{)_{v}}{#1\,\delimsize\vert\,#2}
\DeclarePairedDelimiterX\sipw[2]{(}{)_{w}}{#1\,\delimsize\vert\,#2}

\newcommand{\pair}[2]{\left(\begin{array}{c}\!\!#1\!\!\\ \!\!#2\!\!\end{array}\right)}

\newcommand{\dupN}{\mathbb{N}}

\newcommand{\seq}[1]{(#1_{n})_{n\in\dupN}}

\newcommand{\dupR}{\mathbb{R}}

\newcommand{\dom}{\operatorname{dom}}

\newcommand{\ran}{\operatorname{ran}}

\newcommand{\D}{\mathscr{D}}

\newcommand{\hil}{H}
\newcommand{\lil}{L}

\newcommand{\kil}{K}

\DeclarePairedDelimiterX\sip[2]{(}{)}{#1\,\delimsize\vert\,#2}
\DeclarePairedDelimiterX\siptilde[2]{(}{)_{\!_{\widetilde{A}}}}{#1\,\delimsize\vert\,#2}
\DeclarePairedDelimiterX\sipf[2]{(}{)_{f}}{#1\,\delimsize\vert\,#2}
\DeclarePairedDelimiterX\sipg[2]{(}{)_{g}}{#1\,\delimsize\vert\,#2}
\DeclarePairedDelimiterX\siptw[2]{(}{)_{\tform+\wform}}{#1\,\delimsize\vert\,#2}
\DeclarePairedDelimiterX\set[2]{\{}{\}}{#1\,\delimsize\vert\,#2}
\DeclarePairedDelimiterX\dual[2]{\langle}{\rangle}{#1,#2}
\DeclarePairedDelimiterX\sipa[2]{(}{)_{\!_A}}{#1\,\delimsize\vert\,#2}
\DeclarePairedDelimiterX\sipc[2]{(}{)_{\!_C}}{#1\,\delimsize\vert\,#2}
\DeclarePairedDelimiterX\sipab[2]{(}{)_{\!_{A+B}}}{#1\,\delimsize\vert\,#2}
\DeclarePairedDelimiterX\sipb[2]{(}{)_{\!_B}}{#1\,\delimsize\vert\,#2}

\newcommand{\operator}[2]{\left(\begin{array}{cc}\!\! I &  #1\!\!\\ \!\! #2& I\!\!\end{array}\right)}
\newcommand{\coperator}[2]{\left(\begin{array}{cc}\!\! cI &  #1\!\!\\ \!\! #2& cI\!\!\end{array}\right)}
\newcommand{\boperator}[2]{\left(\begin{array}{cc}\!\! bI &  #1\!\!\\ \!\! #2& bI\!\!\end{array}\right)}
\allowdisplaybreaks
\begin{document}
\title{Adjoint of sums and products of operators in Hilbert spaces}

\author[Z. Sebesty\'en]{Zolt\'an Sebesty\'en}

\address{%
Department of Applied Analysis,\\ E\"otv\"os L. University,\\ P\'azm\'any P\'eter s\'et\'any 1/c.,\\ Budapest H-1117,\\ Hungary}

\email{sebesty@cs.elte.hu}

\author[Zs. Tarcsay]{Zsigmond Tarcsay}

\address{%
Department of Applied Analysis,\\ E\"otv\"os L. University,\\ P\'azm\'any P\'eter s\'et\'any 1/c.,\\ Budapest H-1117,\\ Hungary}

\email{tarcsay@cs.elte.hu}

\subjclass{Primary 47A05, 47A55, 47B25}

\keywords{Closed operator, adjoint, selfadjoint operator, operator product, operator sum, perturbation}


\begin{abstract}
We provide sufficient and necessary conditions guaranteeing  equations $(A+B)^*=A^*+B^*$ and $(AB)^*=B^*A^*$ concerning densely defined unbounded operators $A,B$ between Hilbert spaces. We also improve the perturbation theory of selfadjoint and essentially selfadjoint operators due to Nelson, Kato, Rellich, and W\"ust. Our method involves the range of two-by-two matrices of the form $M_{S,T}=\operator{-T}{S}$ that makes it possible to treat real and complex Hilbert spaces jointly. 
\end{abstract}

\maketitle

\section{Introduction}
Let $\hil$ and $\kil$ be real or complex  Hilbert spaces. A linear operator $A$ from $\hil$ to $\kil$ is called closed if its graph
\begin{equation*}
G(A)=\set{(x, Ax)}{x\in\dom A},
\end{equation*}
is a closed linear subspace of the product Hilbert space $\hil\times\kil$. Furthermore $A$ is called closable if $\overline{G(A)}$ is the graph of an operator. As it is well known, the fact that $A,B$ are closed does not imply that so are $A+B$ and $AB$. More precisely, it is not difficult to give examples of closed operators $A$ and $B$ such that $A+B$ and $AB$ are not even closable (provided they exist at all). It also cannot be expected that the relations
\begin{align}\label{E:ABA+B}
(AB)^*=B^*A^* \qquad \textrm{and} \qquad (A+B)^*=A^*+B^*
\end{align}
hold in general. There are however various conditions on $A,B$ guaranteeing \eqref{E:ABA+B}. For example, a delicate condition for $(AB)^*=B^*A^*$ is that $A$ be bounded or $B$ admit bounded inverse (for proofs see eg. \cite{Birman, Weidmann}). Identity in \eqref{E:ABA+B} concerning the sum is valid if any of the summands is bounded. 

Stronger results can be gained involving the concept of $A$-bounded\-ness (see \cite{HessKato}): an operator $B$ is called $A$-bounded if $\dom A\subseteq \dom B$ and 
\begin{equation}\label{E:A-bounded}
\|Bx\|^2\leq a\|Ax\|^2+b\|x\|^2,\qquad \textrm{for all $x\in\dom A$},
\end{equation}
see eg. \cite{Birman, KATO, Weidmann}. The infimum of all $a\geq0$ for which $b\geq0$ with property \eqref{E:A-bounded} exists is called the $A$-bound of $B$. $A$-boundedness plays a special role also in the perturbation theory of selfadjoint operators. One of the most significant contributions in this direction is the following classical result due to Kato \cite{KATO} and Rellich \cite{rellich}: If $A$ is selfadjoint and $B$ is symmetric and $A$-bounded with $A$-bound less than one then the sum $A+B$ still remains selfadjoint. 

Various extensions of the Kato--Rellich theorem can be found in the literature, see e.g. \cite{Birman, Jorgensen, KATO, Putnam, Weidmann, Wüst}. Our purpose in this paper is to develop a two-by-two operator matrix technique to gain new characterizations of closed, selfadjoint and essentially selfadjoint operators and  to improve the perturbation theory of these operator classes.  We also underline the fact that we do not restrict ourselves to complex Hilbert spaces. In fact, the method we use throughout makes us able to  treat the real and complex cases jointly.

The main tool of our approach are the range of (unbounded) operator matrices  of type $M_{S,T}:=\operator{-T}{S}$. More precisely, considering two operators $S:\hil\to \kil$ and  $T:\kil\to \hil$, the mapping $M_{S,T}$ is defined to be the operator acting on $\hil\times\kil$ with domain $\dom M_{S,T}:=\dom S\times\dom T$, determined by the correspondence
\begin{equation*}
\pair{x}{y}\mapsto \pair{x-Ty}{Sx+y},\qquad x\in\dom S, y\in\dom T. 
\end{equation*}
It turns out that closedness, selfadjointness and essentially selfadjointness of an operator is in a very close connection with some topological properties of the range of $M_{S,T}$, with certain $S,T$. Hence, perturbation type problems of such kind operators can be transferred to perturbations of the range of the matrices $M_{S,T}$. 
\section{Closedness of sums and products}
In our first result we  offer some necessary and sufficient conditions for the operator equation $(A+B)^*=A^*+B^*$ in terms of the operator matrix $M_{A+B,A^*+B^*}$. Observe that no assumptions on closedness of the operators under consideration are made and also the density of $\dom A\cap \dom B$ is omitted from the hypotheses.

\begin{theorem}\label{T:sum}
Let $A,B$ be densely defined operators between two Hilbert spaces $\hil,\kil.$ The following statements are equivalent:
\begin{enumerate}[\upshape (i)]
\item $\dom (A+B)$ is dense in $\hil$ and $(A+B)^*=A^*+B^*$;
\item $G(A+B)^{\perp}\subseteq \ran M_{A+B,A^*+B^*}$
\end{enumerate}
\end{theorem}
\begin{proof}
Assume first that $A+B$ is densely defined and  $(A+B)^*=A^*+B^*$. Then we conclude that 
\begin{align*}
G(A+B)^{\perp}&=\set{(-(A+B)^*v,v)}{v\in \dom (A+B)^*}\\
              &=\set{(-(A^*+B^*)v,v)}{v\in \dom (A^*+B^*)}\\   
              &=\set[\bigg]{\operator{-(A^*+B^*)}{A+B}\pair{0}{v}}{v\in \dom (A^*+B^*)}\\
              &\subseteq \ran M_{A+B,A^*+B^*},
\end{align*}
which proves (ii). Conversely, assume that (ii) and prove first that $A+B$ is densely defined. To this end, let $u\in\dom(A+B)^{\perp}$ and observe immediately that $(u,0)\in G(A+B)^{\perp}$. Hence there exist $x\in\dom (A+B)$ and $v\in\dom (A^*+B^*)$ such that 
\begin{equation*}
\operator{-(A^*+B^*)}{A+B}\pair{x}{v}=\pair{u}{0}.
\end{equation*}
This yields
\begin{align*}
x-(A^*+B^*)v=u\qquad \textrm{and}\qquad (A+B)x+v=0,
\end{align*}
whence $-(A+B)x=v\in\dom (A^*+B^*)$ and $x+(A^*+B^*)(A+B)x=u$. Consequently we see that 
\begin{align*}
0&=\sip{u}{x}=\sip{x}{x}+\sip{(A^*+B^*)(A+B)x}{x}\\
 &=\sip{x}{x}+\sip{(A+B)x}{(A+B)x},
\end{align*}
whence $x=0$, and thus $u=0$, as claimed. To conclude that $(A+B)^*=A^*+B^*$ it suffices to show that $\dom (A+B)^*\subseteq \dom (A^*+B^*)$. Consider therefore $y\in\dom (A+B)^*$ and observe that $(-(A+B)^*y,y)\in G(A+B)^{\perp}$. According to (ii) we may choose  $x\in\dom (A+B)$ and $v\in\dom (A^*+B^*)$ such that 
\begin{equation*}
\operator{-(A^*+B^*)}{A+B}\pair{x}{v}=\pair{-(A+B)^*y}{y}.
\end{equation*}
This yields then 
\begin{align*}
x-(A^*+B^*)v=-(A+B)^*y\qquad \textrm{and}\qquad (A+B)x+v=y,
\end{align*}
hence 
\begin{align*}
0&=\sip{(x,(A+B)x)}{(-(A+B)^*y,y)}\\
 &=\sip{(x,(A+B)x)}{(x-(A^*+B^*)v,(A+B)x+v)}\\
 &=\sip{x}{x}-\sip{x}{(A^*+B^*)v}+\sip{(A+B)x}{v}+\sip{(A+B)x}{(A+B)x}\\
 &=\sip{x}{x}+\sip{(A+B)x}{(A+B)x}.
\end{align*}
Consequently, $x=0$ and therefore $y=v\in\dom(A^*+B^*)$.
\end{proof}
Repeating the above reasoning with minor modifications we can also furnish necessary and sufficient conditions the for identity $(AB)^*=B^*A^*$:

\begin{theorem}
Let $\hil,\kil,\lil$ be Hilbert spaces and let $A,B$ be linear operators from $\kil$ to $\lil$ and $\hil$ to $\kil$, respectively. The following assertions are equivalent:
\begin{enumerate}[\upshape (i)]
\item $\dom (AB)$ is dense in $\hil$ and $(AB)^*=B^*A^*$;
\item $G(AB)^{\perp}\subseteq \ran M_{AB,B^*A^*}$.
\end{enumerate}
\end{theorem}
\begin{proof}
Assume first that $AB$ is densely defined and  $(AB)^*=B^*A^*$. Then
\begin{align*}
G(AB)^{\perp}&=\set{(-(AB)^*v,v)}{v\in \dom (AB)^*}\\
              &=\set{(-B^*A^*)v,v)}{v\in \dom (B^*A^*)}\\   
              &=\set[\bigg]{\operator{-B^*A^*}{AB}\pair{0}{v}}{v\in \dom (B^*A^*)}\\
              &\subseteq \ran M_{AB,B^*A^*},
\end{align*}
hence (i) implies (ii). Assume conversely that (ii) holds true. We are going to prove first that $AB$ is densely defined. Consider $u\in\dom (AB)^{\perp}$ and observe that $(u,0)\in G(AB)^{\perp}$ and therefore that 
\begin{equation*}
\operator{-B^*A^*}{AB}\pair{x}{v}=\pair{u}{0},
\end{equation*}
for some $x\in\dom (AB)$, $v\in\dom (B^*A^*)$. This yields then 
\begin{align*}
x-(B^*A^*)v=u\qquad \textrm{and}\qquad (AB)x+v=0,
\end{align*}
whence $-ABx=v\in\dom(B^*A^*)$ and $u=x+B^*A^*ABx$. Consequently,
\begin{align*}
0=\sip{u}{x}=\sip{x}{x}+\sip{B^*A^*ABx}{x}=\sip{x}{x}+\sip{ABx}{ABx},
\end{align*}
which gives $x=0$ and thus $u=0$, as it is claimed. We see therefore that $B^*A^*\subseteq (AB)^*$. It suffices therefore to prove that $\dom (AB)^*\subseteq \dom (B^*A^*)$. Consider $z\in\dom (AB)^*$ and observe that $(-(AB)^*z,z)\in G(AB)^{\perp}$. According to assertion (ii) we may choose $x\in\dom (AB)$ and $v\in\dom (B^*A^*)$ such that  
\begin{align*}
\pair{-(AB)^*z}{z}=\operator{-B^*A^*}{AB}\pair{x}{v}.
\end{align*}
This yields then 
\begin{align*}
x-B^*A^*v=-(AB)^*z\qquad \textrm{and}\qquad ABx+v=z,
\end{align*}
and hence 
\begin{align*}
0&=\sip{(x,ABx)}{(-(AB)^*z,z)}\\
 &=\sip{(x,ABx)}{(x-B^*A^*v,ABx+v)}\\
 &=\sip{x}{x}-\sip{x}{B^*A^*v}+\sip{ABx}{v}+\sip{ABx}{ABx}\\
 &=\sip{x}{x}+\sip{ABx}{ABx}.
\end{align*}
Consequently, $x=0$ and therefore $z=v\in\dom(B^*A^*)$.
\end{proof}
The next result improves \cite[Proposition 2.3]{Popovici} characterizing closedness of operators by means of the range of the operator matrix $M_{A,A^*}$. 
\begin{theorem}\label{T:closed}
Let $A$ be a densely defined linear operator between two Hilbert spaces $\hil$ and $\kil$. Then the following statements are equivalent:
\begin{enumerate}[\upshape (i)]
\item $A$ is closed.
\item $\ran M_{A,A^*}=\hil\times\kil$.
\item $\set{(-A^*z,z)}{z\in\dom A^*}^{\perp}\subseteq\ran M_{A,A^*}$.
\end{enumerate}
\end{theorem}
\begin{proof}
The following identity is well known for densely defined closed operators $A$:
\begin{equation*}
G(A)\oplus \set{(-A^*z,z)}{z\in\dom A^*}^{\perp}=\hil\times\kil.
\end{equation*}
Consequently, if we assume (i) then for any $(u,v)\in\hil\times\kil$ we can find $x\in\dom A$ and $z\in\dom A^*$ such that 
\begin{align*}
\pair{u}{v}=\pair{x}{Ax}+\pair{-A^*z}{z}=\operator{-A^*}{A}\pair{x}{z}.
\end{align*}
Hence (i) implies (ii). It is clear that (ii) implies (iii). To see that (iii) implies (i) we are going to prove that $A^*$ is densely defined and $A^{**}=A$. With this aim, consider $v\in(\dom A^*)^{\perp}$ and observe that 
\begin{equation*}
(0,v)\in\set{(-A^*z,z)}{z\in\dom A^*}^{\perp}.
\end{equation*}
By (iii) we can choose $x\in\dom A$ and $z\in\dom A^*$ such that 
\begin{equation*}
\pair{0}{v}=\operator{-A^*}{A}\pair{x}{z}.
\end{equation*}
Consequently, 
\begin{equation*}
0=x-A^*z\qquad \textrm{and}\qquad v=Ax+z,
\end{equation*}
which yields $z\in\dom (AA^*)$ and $v=AA^*z+z$. Hence we obtain that 
\begin{align*}
0=\sip{v}{z}=\sip{AA^*z}{z}+\sip{z}{z}=\sip{A^*z}{A^*z}+\sip{z}{z},
\end{align*}
whence $z=0$ and therefore $v=0$, as it is claimed. It is clear now that the closure $A^{**}$ of exists and extends $A$. Our only claim is therefore to show that $A=A^{**}$, or equivalently $\dom A^{**}\subseteq\dom A$. For this purpose, consider $u\in\dom A^{**}$ and observe that 
\begin{equation*}
(u, A^{**}u)\in\set{(-A^*z,z)}{z\in\dom A^*}^{\perp}.
\end{equation*}
Hence, in view of (iii) we can find $x\in\dom A$ and $z\in\dom A^*$ such that 
\begin{equation*}
\pair{u}{A^{**}u}=\operator{-A^*}{A}\pair{x}{z},
\end{equation*}
which follows  
\begin{equation}\label{E:eq1}
u=x-A^*z\qquad \textrm{and}\qquad A^{**}u=Ax+z.
\end{equation}
The first identity  yields $A^*z\in\dom A^{**}$ and hence $A^{**}u=A^{**}x-A^{**}A^*z=Ax-A^{**}A^*z$. This, together with the second formula of \eqref{E:eq1} gives 
$z=-A^{**}A^*z$ whence 
\begin{equation*}
0\leq\sip{z}{z}=\sip{-A^{**}A^*z}{z}\leq0.
\end{equation*}
Consequently, $z=0$ whence $u=x\in\dom A$, as it is claimed. 
\end{proof}
Recall the celebrated von Neumann theorem \cite{vonNeumann, vonNeumann1930} asserting that both $T^*T$ and $TT^*$ are selfadjoint operators, whenever  $T$ is densely defined and closed. The preceding theorem enables us to establish  the converse of that statement (see also \cite{SZ-TZS:reversed}):
\begin{corollary}
Let $T$ be a densely defined linear operator between $\hil$ and $\kil$. The following assertions are equivalent:
\begin{enumerate}[\upshape (i)]
\item $T$ is closed.
\item $T^*T$ and $TT^*$ are both selfadjoint operators in the Hilbert spaces $\hil$ and $\kil$, respectively.
\item $I+T^*T$ and $I+TT^*$ both have full range. 
\item $\set{(-T^*z,z)}{z\in\dom T^*}^{\perp}\subseteq \ran (I+T^*T)\times\ran (I+TT^*)$.
\end{enumerate}
\end{corollary}
\begin{proof}
 The proof  relies on the following identity:
 \begin{equation}\label{E:product}
\operator{-T^*}{T}\operator{T^*}{-T}=\left(\begin{array}{cc}\!\! I+T^*T &  0\!\!\\ \!\! 0& I+TT^*\!\!\end{array}\right)
\end{equation}
In fact, if $T$ is closed then both matrices on the left side of \eqref{E:product} have full range due to Theorem \ref{T:closed}, hence the operator on the right side has full range too. This means that $I+T^*T$ and $I+TT^*$ are both surjective symmetric, hence selfadjoint operators. Thus (i) implies (ii). Assumption  (ii) implies that $I+T^*T$ and $I+TT^*$ are bounded below closed operators with dense range, thus they must be surjective. Hence (ii) implies (iii). Implication (iii)$\Rightarrow$(iv) is straightforward. Finally, formula \eqref{E:product} yields range inclusion
\begin{equation*}
 \ran (I+T^*T)\times\ran (I+TT^*)\subseteq \ran M_{T,T^*},
\end{equation*}
hence implication (iv)$\Rightarrow$(i) follows from Theorem \ref{T:closed}. 
\end{proof}
\begin{corollary}
A densely defined closable operator $T$ between $\hil$ and $\kil$ is closed if and only if $\dom T^{**}\subseteq \ran (I+T^*T)$ and $\ran T^{**}\subseteq\ran (I+TT^*)$.
\end{corollary}
\begin{proof}
Recall the identity 
\begin{equation*}
G(T^{**})=\set{(-T^*z,z)}{z\in\dom T^*}^{\perp}
\end{equation*}
for densely defined closable operators. Thus, by the preceding corollary, $T$ is closed if and only if 
\begin{equation*}
G(T^{**})\subseteq \ran (I+T^*T)\times\ran (I+TT^*),
\end{equation*}
or equivalently, if $\dom T^{**}\subseteq \ran (I+T^*T)$ and $\ran T^{**}\subseteq\ran (I+TT^*)$.
\end{proof}

\section{Perturbation theorems for selfadjoint and essentially selfadjoint operators}

This section is devoted to the perturbation  theory of selfadjoint and essentially selfadjoint operators. The vast majority of the  results to be proved are based on the the next theorem which contains a criterion for a  symmetric operator  operator to be selfadjoint. We also notice that the equivalence of (i) and (ii) below is taken from \cite[Theorem 5.1]{Characterization}, cf. also \cite[Corollary 3.6]{Popovici}. For sake of brevity we also adopt the notation 
\begin{equation*}
M_{S,T}(c):=\coperator{-T}{S}
\end{equation*} 
of \cite{Popovici} for given $S:\hil\to \kil$, $T:\kil\to\hil$ and $c\in\dupR$. 
\begin{theorem}\label{T:selfadjoint}
Let $A$ be a linear operator in the Hilbert space $\hil.$ The following assertions are equivalent:
\begin{enumerate}[\upshape (i)]
 \item $A$ is a (densely defined) selfadjoint operator.
 \item A is symmetric and there exists $c\in\dupR, c\neq0,$ such that 
 \begin{equation*}
\ran M_{A,A}(c)=\hil\times \hil.
\end{equation*}
 \item $A$ is symmetric and there exists $c\in\dupR, c\neq0,$ such that
 \begin{equation*}
G(c^{-1}A)^{\perp}\subseteq \ran  M_{A,A}(c).
\end{equation*} 
\end{enumerate}
\end{theorem}
\begin{proof}
For densely defined closed operators $A$ one has $\ran M_{A,A^*}=\hil\times\hil$, in account of Theorem \ref{T:closed}. Hence, by replacing $A$ by $c^{-1}A$  we conclude that (i) implies (ii).  It is obvious that (ii) implies (iii). Finally, let us assume (iii). We are going to prove first that $A$ is densely defined: pick $u\in\dom A^{\perp}$ and observe that $(u,0)\in G(A)^{\perp}$. Hence there exist $x,y\in \dom A$ such that 
\begin{align*}
\pair{u}{0}=\coperator{-A}{A}\pair{x}{y}=\pair{cx-Ay}{Ax+cy}.
\end{align*}
Consequently, $y=c^{-1}Ax$ and $u=cx+c^{-1}A^2x$. This gives 
\begin{align*}
0=\sip{u}{x}=c\sip{x}{x}+c^{-1}\sip{A^2x}{x},
\end{align*}
whence $x=0=u$, as it is claimed. We see therefore that $A$ is densely defined and symmetric, so $A\subseteq A^*$. We need only to check that $\dom A^*\subseteq\dom A$. To this aim let $v\in\dom A^*$. Since $(-c^{-1}A^*v,v)\in G(c^{-1}A)^{\perp}$ we see by (iii) that there exist $x,y\in \dom A$ such that 
\begin{align*}
\pair{-c^{-1}A^*v}{v}=\coperator{-A}{A}\pair{x}{y}=\pair{cx-Ay}{Ax+cy}.
\end{align*}
Consequently, 
\begin{align*}
0&=\sip[\bigg]{\pair{cx-Ay}{Ax+cy}}{\pair{x}{c^{-1}Ax}}\\ 
 &=c\sip{x}{x}-\sip{Ay}{x}+c\sip{y}{c^{-1}Ax}+c^{-1}\sip{Ax}{Ax}\\
 &=c\sip{x}{x}+c^{-1}\sip{Ax}{Ax},
\end{align*}
whence $x=0$. This yields $v=cy\in\dom A$ and thus we conclude that $A^*=A$, as it is stated. 
\end{proof}
\begin{corollary}
Let $A,B$ be densely defined linear operators in the Hilbert space $\hil$ such that $B\subseteq A \subseteq A^*\subseteq B^*$. Then $A$ is selfadjoint if and only if 
\begin{equation*}
\sip{Ax}{y}=\sip{x}{B^*y},\qquad \forall x\in \dom A,
\end{equation*}
implies $y\in\dom A$ for any $y\in \dom B^*$. 
\end{corollary}
\begin{proof}
Let $(u,y)\in G(A)^{\perp}$, i.e., $\sip{x}{u}+\sip{Ax}{y}=0$ for all $x\in\dom A$. Then 
\begin{equation*}
\sip{Bx}{y}=\sip{Ax}{y}=\sip{x}{-u}, \qquad \textrm{for all $x\in \dom B$},
\end{equation*}
which gives $y\in\dom B^*$ and $B^*y=-u.$ Our assumptions imply  $y\in\dom A$ and  $-u=B^*y=Ay$ which yields
\begin{align*}
\pair{u}{y}=\pair{-Ay}{y}=\operator{-A}{A}\pair{0}{y}\in \ran M_{A,A}(1).
\end{align*}
Theorem \ref{T:selfadjoint} therefore applies. 
\end{proof}
 As a straightforward consequence of Theorem \ref{T:selfadjoint} we get a useful  characterization of essentially selfadjoint operators  in terms of the operator matrix $M_{A,-A}(c)$. Note also immediately that this result generalizes \cite[Theorem 5.1]{Charess}.
\begin{theorem}\label{C:essential}
Let $A$ be a closable (not necessarily densely defined) symmetric operator in the real or complex Hilbert space $\hil$. The following assertions are equivalent: 
\begin{enumerate}[\upshape (i)]
 \item $A$ is essentially selfadjoint.
 \item There exists $c\in\dupR, c\neq0,$ such that 
 \begin{equation*}
\overline{\ran M_{A,A}(c)}=\hil\times\hil.
\end{equation*}
 \item There exists $c\in\dupR, c\neq0,$ such that
 \begin{equation*}
G(c^{-1}A)^{\perp}\subseteq \overline{\ran M_{A,A}(c)}.
\end{equation*} 
\end{enumerate}
\end{theorem}
\begin{proof}
Observe on the one hand that, for $c\in\dupR, c\neq0$ and $x,y\in\dom A$ the symmetry of $A$ yields   
\begin{align*}
\left\|\coperator{-A}{A}\pair{x}{y}\right\|^2=c^2\left\|\pair{x}{y}\right\|^2+\left\|\pair{Ax}{Ay}\right\|^2
\end{align*}
hence $M_{A,A}$ is bounded below by $c^2>0$. On the other hand, denoting by $\overline{A}$ the closure of $A$ one easily verifies that $M_{A,A}(c)$ is closable and its closure equals $M_{\overline{A},\overline{A}}(c)$. Since $M_{\overline{A},\overline{A}}(c)$ is bounded below  by $c^2$ as well, its range is closed. Thus we get 
\begin{align}
\overline{\ran M_{A,A}(c)}=\ran M_{\overline{A},\overline{A}}(c).
\end{align}
We can therefore apply Theorem \ref{T:selfadjoint} to the symmetric operator $\overline{A}$.
\end{proof}

We now turn our attention to the perturbation theory of selfadjoint (resp. essentially selfadjoint) operators. To do so let us  recall first the notion of $A$-boundedness: Given two linear operators $A$ and $B$ in the real or complex Hilbert space $\hil$ we say that $B$ is $A$-bounded if $\dom A\subseteq \dom B$ and there exist $\alpha,\beta\geq0$ such that 
\begin{align}\label{E:Tbounded}
\|Bx\|^2\leq \alpha\|Ax\|^2+ \beta\|x\|^2,\qquad \textrm{for all $x\in \dom A$}.
\end{align}
Note that we receive the same definition of $A$-boundedness if we replace \eqref{E:Tbounded} by 
\begin{align*}
\|Bx\|\leq \alpha'\|Ax\|+ \beta'\|x\|,\qquad \textrm{for all $x\in \dom A$},
\end{align*}
see \cite{KATO}. If $B$ is $A$-bounded then the $A$-bound of $B$ is defined to be the greatest lower bound of the $\alpha$'s satisfying \eqref{E:Tbounded}. An easy application of the closed graph theorem implies that if $A$ is closed and $B$ is closable with $\dom A\subseteq \dom B$ then $B$ must be $A$-bounded, see \cite{Weidmann}. We shall also use the fact that if $A,B$ are closable and $B$ is $A$-bounded then $x_n\to x$ and $Ax_n\to Ax$ imply $x\in\dom B^{**}$ and $Bx_n\to B^{**}x$ for $\seq{x}$ of $\dom A$ and $x$ in $\dom A^{**}$. In particular, $B^{**}$ is $A^{**}$-bounded in that case.  

Our aim in the rest of this section is to provide conditions, involving $A$-bounded\-ness,  which imply selfadjointness (resp., essentially selfadjointness) on the sum between a selfadjoint (resp., essentially selfadjoint) and a symmetric operator. Similar perturbation problems were considered by several authors; cf. eg. \cite{Birman, Jorgensen, KATO, Putnam, Weidmann}. In our approach an essential role is played by the operator matrix $M_{A,A}(c)$.
\begin{theorem}\label{T:perturb1}
Let $A,B$ be symmetric operators in the real or complex Hilbert space $\hil$. Assume that $A$ is essentially selfadjoint and that 
\begin{enumerate}[\upshape a)]
\item $B$ is $A$-bounded,
\item $\|Bx\|^2\leq \|Ax\|^2+\|(A+B)x\|^2+b^2\|x\|^2$ for all $x\in \dom A$ with some $b>0$.
\end{enumerate}
 Then $A+B$ is essentially selfadjoint too. 
\end{theorem}
\begin{proof}
We are going to prove first that $A^{**}+B^{**}$ is essentially selfadjoint. With this aim let $(w,z)\in\ran M_{A^{**}+B^{**},A^{**}+B^{**}}(b)^{\perp}$. By selfadjointness of $A^{**}$ we may choose $x,y\in\dom A^{**}$ such that 
\begin{align*}
\boperator{-A^{**}}{A^{**}}\pair{x}{y}=\pair{w}{z}.
\end{align*}
From assumption a) it follows that $\dom A^{**}\subseteq\dom B^{**}$, hence $x,y\in\dom B^{**}$ and 
\begin{align*}
0&=\sip[\bigg]{\boperator{-A^{**}}{A^{**}}\pair{x}{y}}{\boperator{-(A^{**}+B^{**})}{A^{**}+B^{**}}\pair{x}{y}}\\
&=\sip[\bigg]{\boperator{-(A^{**}+B^{**})}{A^{**}+B^{**}}\pair{x}{y}}{\boperator{-A^{**}}{A^{**}}\pair{x}{y}}.
\end{align*}
Summing up, we get
\begin{align*}
0&=2\left\|\boperator{-A^{**}}{A^{**}}\pair{x}{y}\right\|^2 \\ &\quad +\sip[\bigg]{\boperator{-A^{**}}{A^{**}}\pair{x}{y}}{\left(\begin{array}{cc}\!\! 0 &  -B^{**}\!\!\\ \!\! B^{**}& 0\!\!\end{array}\right)\pair{x}{y}}\\ 
&\quad +\sip[\bigg]{\left(\begin{array}{cc}\!\! 0 &  -B^{**}\!\!\\ \!\! B^{**}& 0\!\!\end{array}\right)\pair{x}{y}}{\boperator{-A^{**}}{A^{**}}\pair{x}{y}}\\
&= 2b^2(\|x\|^2+\|y\|^2)+2(\|A^{**}x\|^2+\|A^{**}y\|^2)\\
&\quad + \sip{bx-A^{**}y}{-B^{**}y}+\sip{A^{**}x+by}{B^{**}x}\\
&\quad+ \sip{-B^{**}y}{bx-A^{**}y}+\sip{B^{**}x}{A^{**}x+by}\\
&=  2b^2(\|x\|^2+\|y\|^2)+2(\|A^{**}x\|^2+\|A^{**}y\|^2)\\
&\quad+ \sip{A^{**}x}{B^{**}x}+\sip{B^{**}x}{A^{**}x}+ \sip{A^{**}y}{B^{**}y}+\sip{B^{**}y}{A^{**}y}\\
&= b^2(\|x\|^2+\|y\|^2)\\ 
&\quad+(b^2\|x\|^2+\|A^{**}x\|^2+\|(A^{**}+B^{**})x\|^2-\|B^{**}x\|^2)\\
&\quad+ (b^2\|y\|^2+\|A^{**}y\|^2+\|(A^{**}+B^{**})y\|^2-\|B^{**}y\|^2)\\
&\geq  b^2(\|x\|^2+\|y\|^2).
\end{align*}
Here in the last estimation we used that inequality b) remains true for $x\in\dom A^{**}$ with $A^{**}, B^{**}$ in place of $A,B$, thanks to $A$-boundedness of $B$. 
We conclude therefore that $x=y=0$ whence $w=z=0$. This follows that $A^{**}+B^{**}$ is essentially selfadjoint. Observe finally that $\dom A^{**}\subseteq\dom (A+B)^{**}$, according to a), again, and hence that $A^{**}+B^{**}\subset (A+B)^{**}$. This implies the essential selfadjointness of $A+B$.   
\end{proof}
\begin{corollary}\label{C:perturb2}
Let $A,B$ be symmetric operators in the Hilbert space $\hil$ with $\dom A\subseteq\dom B$. Assume that $A$ is selfadjoint and that 
\begin{enumerate}[\upshape a)]
\item $A$ is $(A+B)$-bounded,
\item $\|Bx\|^2\leq \|Ax\|^2+\|(A+B)x\|^2+b^2\|x\|^2$, for all $x\in\D$ with some $b>0$. 
\end{enumerate}
Then $A+B$ is selfadjoint too. 
\end{corollary}
\begin{proof}
Observe first that $A,B$ fulfill all conditions of Theorem \ref{T:perturb1}. (Theorem \ref{T:perturb1} a) is satisfied because $A$ is closed and $B$ is closable with $\dom A\subseteq \dom B$.) Thus $A+B$ is essentially selfadjoint in account of Theorem \ref{T:perturb1}. The closedness of $A+B$ is guaranteed by condition a) due to closedness of $A$.
\end{proof}
In the ensuing corollary we establish a symmetric variant of Theorem \ref{T:perturb1} and Corollary  $\ref{C:perturb2}$:
\begin{corollary}\label{C:perturb3}
Let $A,C$ be symmetric operators in $\hil$ with common domain $\D$. Assume that $A$ is $C$-bounded, $C$ is $A$-bounded and that 
\begin{equation*}
\|(A-C)x\|^2\leq \|Ax\|^2+\|Cx\|^2+\gamma\|x\|^2
\end{equation*}
for some $\gamma>0$. Then $A$ is selfadjoint (resp., essentially selfadjoint) if and only if  $C$ is selfadjoint (resp., essentially selfadjoint).
\end{corollary}
\begin{proof}
If we assume that that $A$ is essentially selfadjoint then $B:=C-A$ fulfills all conditions of Theorem \ref{T:perturb1}. Indeed, condition b) is seen immediately. On the other hand, 
\begin{align*}
\|Bx\|\leq \|Ax\|+\|Cx\|\leq (1+\alpha')\|Ax\|+\beta'\|x\|,\qquad x\in\D,
\end{align*}
by $A$-boundedness of $C$, whence we see that a) of Theorem \ref{T:perturb1} is also fulfilled. Consequently, $C=A+B$ is essentially selfadjoint. If $A$ is selfadjoint then each condition of Corollary \ref{C:perturb2} is satisfied due to $C$-boundedness of $A$. Hence $C=A+B$ is selfadjoint. 
\end{proof}

As a corollary we retrieve Kato's result on simultaneous selfadjointness of symmetric operators \cite[Theorem V.4.5]{KATO}:
\begin{corollary}
Let $A,C$ be symmetric operators in $\hil$ with common dense domain $\D$. Assume that  there exist $a,b>0, b<1$ such that 
\begin{equation}\label{E:A-C}
\|(A-C)x\|\leq a\|x\|+b(\|Ax\|+\|Cx\|),\qquad \textrm{for all  $x\in\D$}. 
\end{equation}
Then $A$ is  selfadjoint (resp., essentially selfadjoint) if and only if $C$ is  selfadjoint (resp., essentially selfadjoint). 
\end{corollary}
\begin{proof}
First of all observe that $C$ is $A$-bounded and $A$ is $C$-bounded: Indeed,  \eqref{E:A-C} yields
\begin{equation}\label{E:C<A}
(1-b)\|Cx\|\leq a\|x\|+(1+b)\|Ax\|,\qquad x\in \D.
\end{equation}
whence we see that $C$ is $A$-bounded. That $A$ is $C$-bounded follows by symmetry. 
Choose furthermore $\varepsilon>0$ with $b^2+\varepsilon^2<1$. Then 
\begin{align*}
\|(A-C)x\|^2&\leq  b^2(\|Ax\|^2+\|Cx\|^2) \\
  &\quad+2ab\varepsilon^{-1}\|x\|(\varepsilon\|Ax\|+\varepsilon\|Cx\|)+a^2\|x\|^2\\
  &\leq (b^2+\varepsilon^2)(\|Ax\|^2+\|Cx\|^2)+(a^2+2a^2b^2\varepsilon^{-2})\|x\|^2\\
  &\leq \|Ax\|^2+\|A+Bx\|^2+\gamma\|x\|^2,
\end{align*}
whence  we see that $A,B$ fulfill all conditions of Corollary \ref{C:perturb3}. 
\end{proof}

Nelson in his famous paper \cite{Nelson} proved that if  $A,B$ are commuting symmetric operators on a dense subspace of the (complex) Hilbert space $\hil$ then the essential selfadjointness of $A^2+B^2$ implies essential selfadjointness for both $A$ and $B$ (see \cite[Corollary 9.2]{Nelson} or Corollary \ref{C:Nelson} below for the precise statement). Below we offer two perturbation theorems, inspired by Nelson's result,  in which we prove essential selfadjointness on $A+B$ and $B$ under the weaker condition
\begin{equation*}
\sip{Ax}{By}=\sip{Bx}{Ay},\qquad x,y\in\D
\end{equation*}
instead of requiring the commutation property $ABx=BAx, x\in\D$. Furthermore, our  two-by-two operator matrix technique enables us to extend Nelson's result also for real Hilbert spaces.

Let us start with the following technical lemma:
\begin{lemma}\label{L:AxBy}
Let $A,B$ be symmetric operators in $\hil$ with common (not necessarily dense) domain $\D$ such that 
\begin{equation}\label{E:AxBy}
\sip{Ax}{By}=\sip{Bx}{Ay},\qquad \textrm{for all $x,y\in\D$}.
\end{equation}
Then 
\begin{multline*}
\sip[\bigg]{\operator{-A}{A}\pair{x}{y}}{\operator{B}{-B}\pair{u}{v}}\\
 =\sip[\bigg]{\operator{-B}{B}\pair{x}{y}}{\operator{A}{-A}\pair{u}{v}}
\end{multline*}
for all $x,y,u,v\in\D$.
\end{lemma}
\begin{proof}
Let us compute the left side: 
\begin{multline*}
\sip[\bigg]{\operator{-A}{A}\pair{x}{y}}{\operator{B}{-B}\pair{u}{v}}\\
\begin{aligned}
&=\sip[\bigg]{\pair{x-Ay}{Ax+y}}{\pair{u+Bv}{v-Bu}}\\
&=\sip{x}{u}-\sip{Ay}{u}+\sip{x}{Bv}-\sip{Ay}{Bv}\\
&\quad +\sip{Ax}{v}+\sip{y}{v}-\sip{Ax}{Bu}-\sip{y}{Bu}\\
&=\sip{x}{u}-\sip{y}{Au}+\sip{Bx}{v}-\sip{By}{Av}\\
&\quad+\sip{x}{Av}+\sip{y}{v}-\sip{Bx}{Au}-\sip{By}{u}\\
&=\sip{x}{u+Av}+\sip{y}{v-Au}+\sip{Bx}{v-Au}+\sip{-By}{u+Av}\\
&=\sip[\bigg]{\pair{x-By}{Bx+y}}{\pair{u+Av}{v-Au}}\\
&=\sip[\bigg]{\operator{-B}{B}\pair{x}{y}}{\operator{A}{-A}\pair{u}{v}},
\end{aligned}
\end{multline*}
as it is stated. 
\end{proof}
\begin{theorem}\label{T:perturbAB}
Let $A,B$ be symmetric operators in the real or complex Hilbert space $\hil$ with common (dense) domain $\D$. Assume that $A$ is selfadjoint and that 
\begin{equation*}
\sip{Ax}{By}=\sip{Bx}{Ay},\qquad \textrm{for all $x\in\D$}.
\end{equation*}
Then $B$ and $A+B$ are both essentially selfadjoint on $\D$.
\end{theorem}
\begin{proof}
We prove first that $B$ essentially selfadjoint. With this aim consider $(w,z)\in\ran (M_{B,B})^{\perp}$. By selfadjointness of $A$ we can find $u,v\in\D$ such that 
\begin{equation*}
\operator{A}{-A}\pair{u}{v}=\pair{w}{z}.
\end{equation*}
Consequently, by Lemma \ref{L:AxBy}
\begin{multline*}
0=\sip[\bigg]{\operator{A}{-A}\pair{u}{v}}{\operator{-B}{B}\pair{x}{y}}\\
 =\sip[\bigg]{\operator{B}{-B}\pair{u}{v}}{\operator{-A}{A}\pair{x}{y}},
\end{multline*}
 for any $x,y\in\D$. Hence $\operator{B}{-B}\pair{u}{v}\in\ran (M_{A,A})^{\perp}=\{0\}$. We get therefore 
 \begin{align*}
  0=\left\|\operator{B}{-B}\pair{u}{v}\right\|^2=\|u\|^2+\|v\|^2+\|Bu\|^2+\|Bv\|^2,
 \end{align*}
 that yields $u=v=0$, whence $w=z=0$. Consequently, $B$ is essentially selfadjoint due to Theorem \ref{C:essential}. Observe on the other hand that 
 \begin{align*}
 \sip{Ax}{(A+B)y}=\sip{Ax}{Ay}+\sip{Bx}{Ay}=\sip{(A+B)x}{Ay}
 \end{align*}
 holds according to our hypotheses. The essential selfadjointness of $A+B$ follows therefore due to the first part of the proof, replacing $B$ by $A+B$.
\end{proof}
The essentially selfadjoint variant of the preceding theorem reads as follows:
\begin{theorem}\label{T:essentialABBA}
Let $A,B$ be symmetric operators in the real or complex Hilbert space $\hil$ such that $A$ is essentially selfadjoint and $B$ is $A$-bounded. If 
\begin{equation}\label{E:ABBA}
\sip{Ax}{By}=\sip{Bx}{Ay},\qquad \textrm{for all $x\in\dom A$}
\end{equation}
then $A+B$ and $B$ are essentially selfadjoint on $\dom A$.
\end{theorem}
\begin{proof}
We conclude by $A$-boundedness that $\dom A^{**}\subseteq \dom B^{**}$ and $x_n\to x$ and $Ax_n\to Ax$ imply $Bx_n\to Bx$ for $\seq{x}$ of $\dom A$ and $x$ in $\dom A^{**}$. Hence \eqref{E:ABBA} remains valid also for $x\in\dom A^{**}=:\D$ with $A,B$ replaced by $A^{**}, B^{**}$, respectively. Theorem \ref{T:perturbAB} implies therefore that $A^{**}+B^{**}$ as well as the restriction of $B^{**}$ to $\dom A^{**}$ are essentially selfadjoint. The desired statement follows from the observations
\begin{equation*}
G(A^{**}+B^{**})\subseteq \overline{G(A+B)}\qquad\textrm{and}\qquad G(B^{**}\upharpoonright\dom A^{**})\subseteq \overline{G(B)} 
\end{equation*}
which are again due to $A$-boundedness.
\end{proof}
\begin{corollary}
Let $A,B$ be symmetric operators in $\hil$ with $\D:=\dom A\subseteq\dom B$. Assume furthermore  $\D\subseteq \dom (AB)\cap\dom(BA)$, $ABx=BAx$ for all $x\in\D$, and that any of the following two conditions is satisfied:
\begin{enumerate}[\upshape a)]
\item $A$ is selfadjoint;
\item $A$ is essentially selfadjoint and $B$ is $A$-bounded.
\end{enumerate}
Then  $B$ and $A+B$ are both essentially selfadjoint on $\D$.  
\end{corollary}
\begin{proof}
It is easy to verify that $A,B$ satisfy  the conditions of Theorem \ref{T:perturbAB} and Theorem \ref{T:essentialABBA}, respectively.
\end{proof}
\begin{corollary}
Assume that $A$ is symmetric and   $A^m$ is essentially selfadjoint for some integer $m$. If $p$ is a polynomial with real coefficients and of degree $\leq m$  then $p(A)$ is essentially selfadjoint on $\dom (A^m)$.  
\end{corollary}
\begin{proof}
It is obvious that the symmetric operator $B:=p(A)$ fulfills \eqref{E:ABBA} with $A^m$ in place of $A$. To see that $p(A)$ is $A^m$-bounded it suffices to show that   $A^{n}$ is $A^{n+1}$-bounded for each integer $n< m$. We proceed by induction. The case $n=2$ follows easily from the Cauchy--Schwarz and the arithemetic-geometric mean inequalities:
\begin{align*}
\|Ax\|^2=\sip{A^2x}{x}\leq \frac12\|A^2x\|^2+\frac12\|x\|^2,\qquad \textrm{for all $x\in\dom (A^2)$}. 
\end{align*}
Suppose now $\|A^{n-1}x\|^2\leq \alpha\|A^{n}x\|^2+\beta\|x\|^2$ for certain $\alpha,\beta$ and for all $x\in\dom(A^{n})$. Choosing $\gamma>0$ with $2\gamma^2>\alpha$ we get, just as in case $n=2$,
\begin{align*}
\|A^{n}x\|^2&\leq \frac{\gamma^2}{2}\|A^{n+1}x\|^2+\frac{1}{2\gamma^2}\| A^{n-1}x\|^2\\
&\leq \frac{\gamma^2}{2}\|A^{n+1}x\|^2+\frac{\alpha}{2\gamma^2}\| A^{n}x\|^2+ \frac{\beta}{2\gamma^2}\|x\|^2
\end{align*}
for $x\in\dom (A^{n+1})$, whence, by rearrangement we conclude that $$\|A^{n}x\|^2\leq \alpha'\|A^{n+1}x\|^2+\beta'\|x\|^2$$ for suitable $\alpha',\beta'.$ Theorem \ref{T:essentialABBA} completes then the proof.
\end{proof}
We conclude the paper with a result \cite[Corollary 9.2]{Nelson} due to Nelson. We also mention that Nelson's original result concerned only with complex Hilbert spaces.
\begin{corollary}\label{C:Nelson}
Let $A,B$ be symmetric operators in the real or complex Hilbert space $\hil$. Let $\D$ be a dense subset of $\hil$ such that $\D$ is contained in the domain of $A,B, AB, BA, A^2$ and $B^2$ and such that $ABx=BAx$ for all $x\in\D$. If the restriction of $A^2+B^2$ to $\D$ is essentially selfadjoint then the restrictions of $A$ and $B$ to $\D$ are essentially selfadjoint as well.  
\end{corollary}
\begin{proof}
Setting $C:=(A^2+B^2)\upharpoonright\D$ we conclude that 
\begin{align*}
\sip{Ax}{Cy}=\sip{Cx}{Ay}\qquad \textrm{and}\qquad \sip{Bx}{Cy}=\sip{Cx}{By}
\end{align*}
for all $x,y\in\D$. On the other hand,  both $A,B$ are $C$-bounded according to the following estimation:
\begin{align*}
\|Ax\|^2+\|Bx\|^2=\sip{Cx}{x}\leq \frac12\big(\|Cx\|^2+\|x\|^2\big),\qquad x\in\D.
\end{align*}
Theorem \ref{T:essentialABBA} therefore applies. 
\end{proof}
\begin{remark}
Note that  the assumptions of Nelson's result above yield essential selfadjointness for $aA+bB$ for any $a,b\in \dupR$. To see this we can repeat the arguments of the preceding proof: 
\begin{align*}
\sip{(aA+bB)x}{Cy}=\sip{Cx}{(aA+bB)y},\qquad x,y\in\D,
\end{align*}
and $aA+bB$ is $C$-bounded since $A$ and $B$ are so.
\end{remark}
\bibliographystyle{abbrv}

\end{document}